%% file: elementary_main.tex
\documentclass[11pt]{article}
\usepackage[margin =1in]{geometry}
\usepackage{amssymb,amsthm,amsmath}
\usepackage{enumerate}
\usepackage{hyperref}
\usepackage{cite}
\usepackage[capitalize]{cleveref}
\usepackage{enumitem}
\usepackage{color}
\usepackage{cancel}
\usepackage{pgf}
\usepackage{svg}
\usepackage{pgfplots}
\usepackage{import}

\usepackage[utf8]{inputenc} 
\usepackage[T1]{fontenc}    
\usepackage{hyperref}       
\usepackage{url}            
\usepackage{booktabs}       
\usepackage{amsfonts}       
\usepackage{nicefrac}       
\usepackage{microtype}
\usepackage{multirow}
\usepackage{xfrac}
\usepackage{todonotes}
\usepackage{titlesec}
\usepackage{caption}
\usepackage{subcaption}
\usepackage{algorithm}
\usepackage{algorithmic}
\usepackage{tcolorbox}
\tcbset{
  myboxstyle/.style={
    colback=blue!3!white,       
    colframe=blue!50!black,     
    boxrule=1pt,              
    arc=3pt,                    
    left=6pt, right=6pt, top=4pt, bottom=4pt,
  }
}

\titleformat{\subsubsection}[runin]
{\normalfont\normalsize\bfseries}{\thesubsubsection}{1em}{}

\usepackage{graphicx}

\usepackage{appendix}
\numberwithin{equation}{section}

\newcommand{\simplex}[1]{\Delta_{#1}}
\newcommand{\smax}{\sigma_\mathtt{max}}
\newcommand{\defeq}{:=}
\newcommand{\norm}[1]{\left \| #1 \right \|}
\newcommand{\R}{\mathbb{R}}
\newcommand{\ip}[1] {\left\langle #1 \right\rangle }

\newtheorem{thm}{Theorem}[section]
\newtheorem{theorem}{Theorem}[section]

\newtheorem{proposition}[thm]{Proposition}

\newtheorem{lemma}[thm]{Lemma}

\crefname{claim}{claim}{claims}
\Crefname{claim}{Claim}{Claims}
\crefname{lem}{lemma}{lemmas}
\Crefname{lem}{Lemma}{Lemmas}
\crefname{algorithm}{algorithm}{algorithms}
\Crefname{algorithm}{Algorithm}{Algorithms}

\theoremstyle{remark}

\definecolor{blue}{rgb}{0,0,0}

\begin{document}

	\title{An Elementary Proof of the Near Optimality\\ of LogSumExp Smoothing}

	  \author{Thabo Samakhoana\footnote{Johns Hopkins University, Department of Applied Mathematics and Statistics, \url{tsamakh1@jhu.edu}} \qquad Benjamin Grimmer\footnote{Johns Hopkins University, Department of Applied Mathematics and Statistics, \url{grimmer@jhu.edu}}}

	\date{}
	\maketitle

	\input{abstract}

    \input{bodyText}

    \paragraph{Acknowledgments.} This work was supported in part by the Air Force Office of Scientific Research under award number FA9550-23-1-0531. Benjamin Grimmer was additionally supported as a fellow of the Alfred P. Sloan Foundation.

    {\small
    \bibliographystyle{unsrt}
    \bibliography{bibliography}
    }

    \input{appendix}
\end{document}

%% file: abstract.tex
\begin{abstract}
    We consider the design of smoothings of the (coordinate-wise) max function in $\mathbb{R}^d$ in the infinity norm. The LogSumExp function $f(x)=\ln(\sum^d_i\exp(x_i))$ provides a classical smoothing, differing from the max function in value by at most $\ln(d)$. We provide an elementary construction of a lower bound, establishing that every overestimating smoothing of the max function must differ by at least $\sim 0.8145\ln(d)$. Hence, LogSumExp is optimal up to small constant factors. However, {\color{blue} we provide strictly stronger smoothings showing the entropy-based LogSumExp approach is not exactly optimal.} In small dimensions, we propose exactly optimal smoothings, attaining our lower bound.
\end{abstract}

%% file: bodyText.tex
  \section{Introduction}

    We consider the task of smoothing the coordinate-wise max function $\smax(x) = \max_{i=1,\dots, d} x_i$ with respect to the infinity norm in $d$ dimensions. Such smoothings play an important role in the acceleration of nonsmooth optimization algorithms~\cite{Nesterov2005} and recently in machine learning~\cite{Martins2016,Asadi2017}. In both of these settings, the function $f_\mathtt{LSE}(x) = \ln(\sum^d_i\exp(x_i))$ has played a central role as the default ``best'' choice in the engineering of algorithms.
    
    The canonical choice of the LogSumExp function $f_\mathtt{LSE}$ (also known as the ``softmax'') is a convex overestimator of $\smax$ that is $1$-smooth with respect to the infinity norm:
    \begin{equation} \label{eq:def-smoothness}
        \|\nabla f(x) - \nabla f(y)\|_1 \leq \|x-y\|_\infty \qquad \forall x,y\in\mathbb{R}^d.
    \end{equation}
    Moreover, $f_\mathtt{LSE}$ closely approximates the max function, having
    $$ \|f_\mathtt{LSE} - \smax\|_\infty = \ln(d) $$
    where $\|f-\sigma\|_\infty \defeq \sup_{x\in\mathbb{R}^d} |f(x) - \sigma(x)|$.
    Generally, we say that $f$ is a $\delta$-smoothing with respect to the infinity norm of the max function if $f$ is convex and $1$-smooth w.r.t.~$\|\cdot\|_\infty$ with $\|f-\smax\|_\infty \leq \delta$. Hence, $f_\mathtt{LSE}$ is a $\ln(d)$-smoothing of $\smax$.

    Such $\delta$-smoothings in the infinity norm play an important role for both of the applications mentioned above. Below we discuss these two core motivations:
    
    \noindent {\it (i) In nonsmooth optimization,} a common form for a problem to take is minimizing a maximum of convex functions
    $$ \min_{y\in \mathbb{R}^n} \max_{i=1,\dots, d} g_i(y) = \smax\circ g (y) $$
    where $g(y) = [g_1(y),\dots, g_d(y)]^T$. For general convex nonsmooth optimization, any first-order subgradient method seeking an $\epsilon$-minimizer requires at least $\Omega(1/\epsilon^2)$ subgradient oracle queries. However, if each individual component objective $g_i$ is $M$-Lipschitz and $L$-smooth with respect to the Euclidean norm in $\mathbb{R}^n$, then accelerated optimization algorithms from smooth convex optimization can be leveraged. In particular,~\cite{Beck2012} considered minimizing the smooth convex relaxation
    $$ \min_{y\in\mathbb{R}^n} \ \left(\frac{\epsilon}{2\delta}f\right) \circ \left(\frac{2\delta}{\epsilon} g\right)\left(y\right)$$
    for a $\delta$-smoothing $f$ of $\smax$. Since $f$ is 1-smooth with respect to the infinity norm and each individual $g_i$ is smooth and Lipschitz, the objective $y\mapsto \frac{\epsilon}{2\delta} f\circ \frac{2\delta}{\epsilon} g\left(y\right)$ above must be $(L+2\delta M^2/\epsilon)$-smooth with respect to the Euclidean norm in $\mathbb{R}^n$~\cite[Proposition 4.1]{Beck2012}. Further, it differs from the original objective by at most $\epsilon/2$. As a result, by applying an accelerated algorithm for smooth convex optimization like Nesterov's method~\cite{Nesterov1983method} to produce an $\epsilon/2$-minimizer, an $\epsilon$-minimizer to the original problem can be produced with a first-order oracle complexity of at most
    \begin{equation}\label{smoothed-rate}
       \sqrt{\frac{(L+2\delta M^2/\epsilon)\|y_0-y_\star\|^2}{\epsilon/2}} = \mathcal{O}\left(\frac{\sqrt{L\epsilon +\delta M^2}\|y_0-y_\star\|}{\epsilon}\right). 
    \end{equation}
    Note the acceleration in the dependence on $\epsilon$ in the above $\mathcal{O}(1/\epsilon)$ rate over the general $\Omega(1/\epsilon^2)$ limit. Hence, $\delta$-smoothings provide a direct reduction for accelerating structured nonsmooth optimization. Further, improvements in $\delta$ directly improve oracle complexities.

    \noindent {\it (ii) In neural network design,} especially modern networks using the attention/transformer mechanism~\cite{Vaswani2017}, it is often useful for the output of a layer in the network to be a vector of weights for a probability distribution. For example, this can correspond to the likelihood of outputting a given token.    
    That is, one would like a mapping $g\colon \mathbb{R}^d\rightarrow \Delta_d$ where $\Delta_d$ is the probability simplex in dimension $d$. Gradients $\nabla f \colon \mathbb{R}^d\rightarrow \Delta_d$ of convex smoothings of the max function provide such a construction with desirable properties like monotonicity and Lipschitzness. Hence, there has been substantial interest in the machine learning literature in such smoothings~\cite{Asadi2017,Martins2016}.
    Since the outputs of the gradient map are probabilities (living in the simplex $\Delta_d$), the typical norm to endow them with is the one-norm. Dually, the natural corresponding norm for the primal space is the infinity norm, as we study here. Note that by convexity, any $\delta$-smoothing $f$ has that
    \begin{equation}
        \langle \nabla f(x), x\rangle \geq \smax(x) - 2\delta \qquad \forall x\in\mathbb{R}^d \label{eq:Expectation-Guarantee}
    \end{equation}
    where $\langle\cdot,\cdot\rangle$ denotes the standard Euclidean inner product,
    i.e., the expected value of sampling an element of $x$ with probabilities $\nabla f(x)$ concentrates near the vector's maximum value. See~\cite{Epasto2020} for a discussion of the potential importance of gains in $\delta$.

    For the task of smoothing in the infinity norm, in either of these disciplines, the LogSumExp function has emerged as the canonical standard choice. The fundamental nature of this choice is best seen by considering a dual perspective. Generally, convex conjugates $f^*(x) = \sup_{\lambda\in\mathbb{R}^d} \langle \lambda, x\rangle - f(\lambda)$ provide a duality between smooth convex functions and strongly convex functions. A Nesterov-style smoothing~\cite{Nesterov2005} of the max function is given by such a conjugate as
    \begin{equation}
        f(x) = \sup_{\lambda\in\Delta_d} \langle \lambda, x\rangle - h(\lambda) \label{eq:Nesterov-smoothing}
    \end{equation}
    for any $1$-strongly convex ``distance'' function $h$ on $\Delta_d$ in the one-norm. From this point of view, designing a smoothing corresponds to designing a strongly convex function on the simplex. If one takes $h(\lambda) = \sum_{i=1}^d \lambda_i\ln(\lambda_i)$ as the entropy function, a well-known strongly convex ``distance'' for the probability simplex, then the resulting smoothing $f$ is exactly $f_\mathtt{LSE}$. Indeed, this development is at the origin of the LogSumExp function's usage in nonsmooth optimization~\cite{Nesterov2005}.
    Despite their ubiquity, the fundamental boundaries of what smoothings can attain remain underexplored.

    In this work, we derive limits on the process of designing smooth approximations of the coordinate-wise max function. That is, we seek to derive lower bounds on the infinite-dimensional optimization problem of finding the best $\delta$-smoothing in $d$ dimensions
    $$ \delta^\star(d) \defeq \begin{cases} \min_f &\|f-\smax\|_\infty\\
        \mathrm{s.t.} & f\colon \mathbb{R}^d\rightarrow \mathbb{R} \text{ is convex and $1$-smooth in $\|\cdot\|_\infty$}. \end{cases} $$
    Equivalently, one could consider optimizing only over overestimators of $\smax$, as $f_\mathtt{LSE}$ does, denoted
    $$ \delta^\star_\mathtt{over}(d) \defeq \begin{cases} \min_f &\|f-\smax\|_\infty\\
        \mathrm{s.t.} & f\colon \mathbb{R}^d\rightarrow \mathbb{R} \text{ is convex and $1$-smooth in $\|\cdot\|_\infty$}\\
        & f(x) \geq \smax(x) \ \forall x\in\mathbb{R}^d. \end{cases} $$
    By adding simple shifts, it is immediate that $\delta^\star(d) = \delta^\star_\mathtt{over}(d)/ 2$. The existing results for the LogSumExp function $f_\mathtt{LSE}$  discussed above then establish an upper bound for this problem of $\delta^\star(d) \leq 0.5 \ln(d)$ for all $d$. We consider the complementary question of lower bounds.

    \subsection{Our Contributions}
    A lower bound on $\delta^\star(d)$ of $\ln(d)/8$ can be extracted from two related lower bounds in the literature.
    Epasto et al.~\cite{Epasto2020} provided a lower bound via limits on any Lipschitz mappings satisfying~\eqref{eq:Expectation-Guarantee}. Their approach relies on a game-theoretic argument establishing the existence of a Nash equilibrium from an application of Glicksberg's theorem and a careful compactification approach.
    \begin{proposition}[Theorem 4.4, Epasto et al.~\cite{Epasto2020}]\label{prop:epasto}
        For any dimension $d=2^{2^k}$ for an integer $k\geq 0$, one has
        $$ \delta^\star(d) \geq \frac{1}{8}\ln(d).$$
    \end{proposition}
    \noindent Alternatively, via a reduction to a statistical prediction game of~\cite{Cesa1997} in online learning, one can derive an asymptotically matching lower bound for all $d$ (proven in Appendix~\ref{app:stat_reduction} for completeness).
    \begin{proposition} \label{prop:reduction_result}
        For any dimension $d\geq 2$, one has
        $$ \delta^\star(d) \geq  \left(\frac{1}{8} + o(1)\right)\ln(d).$$
    \end{proposition}
    
    These bounds leave a factor of four gap with the LogSumExp upper bound of $0.5\ln(d)$ and rely on nontrivial equilibrium existence and statistical arguments, respectively. As our main result, we further improve on this, deriving stronger lower bounds directly via an elementary constructive approach relying only on classical inequalities for smooth convex functions and a combinatorial recurrence. {\color{blue}At its core, our analysis works inductively by considering elementary inequalities between certain sparse points $x^{(j_{0})}, x^{(j_{1})}, \dots, x^{(j_{k})} \in \R^d$, each with nonzero components only in the first $j_\ell$ entries, for certain subsets $\{j_0, j_1, \dots, j_{k-1}, j_k\} \subseteq \{1, \dots, d\}$. For each such subset, we derive a lower bound. Optimizing over these bounds, we arrive at the following maximal partition sum quantity
    \begin{equation} \label{eq:gamma-definition}
    \gamma(d) \defeq \begin{cases} \max\limits_{\{j_0, j_1, \dots, j_{k-1}, j_k\} \subseteq \{1, \dots, d\}} &\sum_{\ell=1}^{k} \left(1-\frac{j_{\ell-1}}{j_\ell}\right)^2\\
        \mathrm{s.t. }& 1=j_0< j_1<\dots< j_{k-1}< j_k = d.\end{cases}
    \end{equation}
    Section~\ref{subsec:structured-proofs} shows this quantity can be understood as optimizing over a structured family of proofs to build the strongest lower bound.
    }
    
    Our main result shows that this provides a lower bound on the optimal smoothing gap $\delta^\star(d)$. Our proof of Theorem~\ref{thm:main_result} in fact provides a stronger lower bound, applying to any $p$-norm smoothing.
    \begin{theorem}\label{thm:main_result}
        For any dimension $d\geq 2$, one has
        $$ \delta^\star(d) \geq  \gamma(d).$$
        Further, $\lim_{d\rightarrow \infty} \frac{\gamma(d)}{\ln(d)} =2\beta(1-\beta) \geq 0.40726$ where $\beta$ is the unique root of $2\beta \ln \beta - \beta + 1$ in $(0,1)$. 
    \end{theorem}
    
    Given the $0.5\ln(d)$ upper bound due to the LogSumExp function, this theorem determines the exact optimal gap down asymptotically to a narrow numerical range of $[0.40726, 0.5]$.
    The closeness of these bounds provides a case for the strong performance of the LogSumExp smoothing: It leaves at most a factor of 20\% ``on the table'' in terms of performance. 
    
    However, although nearly matching our bound, $f_\mathtt{LSE}$ fails to attain our lower bound. As a final result, we examine the tightness of our elementary lower bound. {\color{blue} For any dimension $d$, we provide a new smoothing, strictly outperforming $0.5\ln(d)$.} This disproves the optimality of $f_\mathtt{LSE}$ for {\color{blue} every} $d$. In the case of $d=2,3$, our lower bound construction via $\gamma(d)$ is exactly tight, proven by an exactly optimal smoothing attaining $\delta^\star(d)$.
    \begin{theorem}\label{thm:exact_low_dim}
        {\color{blue} For any dimension $d\geq 2$, one has $\delta^\star(d)< 0.5 \ln(d)$.} For $d=2$ and $d=3$, the optimal smoothing value is exactly $\delta^\star(d) = \gamma(d)$.
    \end{theorem}
    The use of such further optimized smoothings offers constant gains in worst-case guarantees for accelerated nonsmooth optimization directly by the previously discussed reduction to smooth optimization techniques of~\cite{Beck2012}. Smoothings attaining our theoretical lower bound saturate what is possible by such reductions, meaning any further progress in accelerated nonsmooth optimization will require approaches distinct from existing smoothing frameworks.
    Resolving the exact optimal smoothing $f$ for $d\geq 4$ is left as an interesting future direction. 

    \subsection{Related Works}

    \noindent {\bf Constructions of Smoothings.} Moreau envelopes~\cite{Moreau1965} are perhaps the most classic form of smoothing in optimization. General recipes for constructing smoothings have been extensively explored. Nesterov~\cite{Nesterov2005} does so via dualizations of the form~\eqref{eq:Nesterov-smoothing}. Beck and Teboulle~\cite{Beck2012} provided an alternative asymptotic smoothing approach applicable for smoothing any sublinear function. Epasto et al.~\cite{Epasto2020} proposed a general piecewise linear form for constructing gradients of a smoothing of the max function, yielding a piecewise quadratic family of $f$. In areas of machine learning, alternatives to the LogSumExp/softmax function have been of recent interest. The sparsemax~\cite{Martins2016}, defined by projections onto the simplex and corresponding to the Moreau envelope of the max function, typically has sparse gradients, which can improve computation and interpretability. In reinforcement learning, the mellowmax smoothing~\cite{Asadi2017}, defined as a translation of the LogSumExp, has found widespread interest.

    \noindent {\bf Prior Lower Bound Theory for Smoothings.} As discussed above, lower bound proofs can be extracted from the literature~\cite{Cesa1997,Epasto2020} via the game theoretic or statistical reductions in Propositions~\ref{prop:epasto} and~\ref{prop:reduction_result} respectively. Our work here circumvents the need for such heavy machinery by using entirely constructive, elementary proof techniques. Moreover, our tailored, direct approach yields tighter lower bounds, being exact for $d=2,3$ at least.

    If one instead considers optimal smoothings with respect to the two-norm, requiring instead that
    $$ \|\nabla f(x) - \nabla f(y)\|_2 \leq \|x-y\|_2 \qquad \forall x,y\in \mathbb{R}^d,$$
    then a theory for exactly optimal smoothings of general sublinear functions was given by the authors in~\cite{samakhoana2025optimal}. In this case, the optimal difference between a smooth convex $f$ and $\smax$ is bounded by a constant for all $d$ (in particular being at most $1/4$, see~\cite[Section 4.2.3]{samakhoana2025optimal}). As a result, $f_\mathtt{LSE}$ is fundamentally suboptimal not just by constants but in a big-O for two-norm smoothing.

    \noindent {\bf Complexity Lower Bound Theory for Optimization.} Often, the complexity theory for first-order methods is independent of the ambient problem dimension $d$. As exceptions to this, Guzman and Nemirovski~\cite{Guzman2015} showed a log term for H\"older smooth convex optimization over $p$-norm balls outside the classical Euclidean setting. Carmon and Hinder~\cite{Carmon2025} showed log factors in problem dimension are a fundamental price of adaptivity for stochastic optimization when certain problem parameters are unknown. For the particular setting of minimizing a finite maximum considered here, our Theorem~\ref{thm:main_result} establishes a smoothing barrier: no smoothing-based reduction, i.e.~\eqref{smoothed-rate}, can have a dimension-independent complexity or any dependence better than $\mathcal{O}(\sqrt{\ln(d)})$.
    
    In online convex optimization, algorithms often must have regret scale at least with $\Omega(\sqrt{\ln(d)})$. This is proven by a reduction to coin flipping games~\cite{Cesa1997} and attained up to constant factors by approaches using entropic regularization, dual to the LogSumExp function. See the surveys~\cite[Section 2.5]{Shalev2012} and~\cite[Section 7]{Orabona2019AMI}. Our proof of Proposition~\ref{prop:reduction_result} reduces to this prior art.

    \noindent {\bf Generalization to the Maximum Eigenvalue Function.} Lastly, we note our theory immediately provides lower bounds on smoothings of the maximum eigenvalue function $\sigma_\lambda(X) = \lambda_\mathtt{max}(X)$ defined for all symmetric matrices $X\in\mathbb{R}^{d\times d}$ with respect to the spectral norm. This follows by noting that for any $x\in\mathbb{R}^d$, $\sigma_\lambda(\mathrm{diag}(x)) = \smax(x)$. As a result, any spectral-norm $1$-smoothing of $\sigma_\lambda$ must differ by at least $\gamma(d)$ somewhere. This is the optimal order as $f(X) = \ln(\sum_{i}^d \exp(\lambda_i(X)))$ provides a 1-smooth overestimator of $\sigma_\lambda$ differing by at most $\ln(d)$. This generalization of the $f_\mathtt{LSE}$ smoothings has played a role in solving semidefinite programs, for example, using the radial framework of Renegar~\cite{Renegar2019} and Grimmer~\cite{Grimmer2021-part2}.

    {\color{blue}
    \subsection{The Lower Bound $\gamma(d)$ as Optimization over Structured Proofs}\label{subsec:structured-proofs}

    Before formally proving Theorem~\ref{thm:main_result}, here we provide an informal progression of ideas leading to $\gamma(d)$ as a plausible candidate lower bound on $\delta^\star(d)$. To build lower bounding inequalities, we consider certain inequalities which must be satisfied for any function that is convex and $1$-smooth in $\|\cdot\|_\infty$. For a differentiable function $f$, these two properties are equivalent to the following being nonnegative for every $x,y\in\mathbb{R}^d$,
    \begin{equation*}
    Q_{x,y} \defeq f(x) - f(y) -\ip{\nabla f(y), x-y} - \frac{1}{2}\|\nabla f(x) - \nabla f(y)\|^2_1 \geq 0.
    \end{equation*}
    See~\cite[Theorem 5.8]{beck2017first}. We aim to construct our lower bound by considering well-chosen combinations of these inequalities at various points $x$ and $y$.

    In particular, we aim to show that some points $y^{(d)}$ and $y^{(1)}$ must have a large difference in $f$ values (larger than $\smax$). This suffices to provide a lower bound as
    $$ \left[f(y^{(d)}) - f(y^{(1)})\right] - \left[\smax(y^{(d)}) - \smax(y^{(1)})\right] \leq 2\|f-\smax\|_\infty.$$
    So showing $\left[f(y^{(d)}) - f(y^{(1)})\right] - \left[\smax(y^{(d)}) - \smax(y^{(1)})\right] - 2\gamma \geq 0$ would prove $\|f-\smax\|_\infty\geq \gamma$.
    
    We prove such a nonnegativity result by introducing intermediate points $y^{(1)},\dots,y^{(d)}$ between which we will consider the nonnegative $Q_{\cdot,\cdot}$ quantities. In particular, we restrict attention to scaled points $y^{(j)} = \alpha x^{(j)}$ with $x^{(j)}= (1/j,\dots,1/j,0,\dots,0)$, having $j$ nonzero entries followed by zeros. In our formal proof below, we will take a limit as $\alpha$ tends to infinity, but for motivating our approach, considering a fixed large value will suffice. For ease, denote $ Q_{y^{(i)},y^{(j)}}$ by $Q_{i,j}$.

    Given these points, one may hope for a proof of the target nonnegativity that selects nonnegative multipliers $\lambda_{i,j}$ for each $i>j$ such that the following identity holds
    \begin{equation} \label{eq:target-identity}
         \sum_{i>j} \lambda_{i,j} Q_{i,j} = \left[f(y^{(d)}) - f(y^{(1)})\right] - \left[\smax(y^{(d)}) - \smax(y^{(1)})\right] - 2\gamma
    \end{equation}
    for any values of $f(\cdot)$ and $\nabla f(\cdot)$ at the points $y^{(j)}$. Since the left-hand side is a sum of nonnegative quantities, this identity would establish the right-hand side is nonnegative.
    Optimizing over such proofs, the strongest lower bound is given by the $\lambda$ yielding the maximum value of $\gamma$.

    As a further simplification, one may guess that the gradients at the points $y^{(j)}$, for $\alpha$ large enough, converge towards $(1/j,\dots,1/j,0,\dots,0)$. Indeed, in~\eqref{eq:gradient-derivation}, we will find this can be taken without loss via a symmetry argument. Under this asymptotic simplification, we have for $i>j$ that
    \begin{equation*}
    Q_{i,j} = \left[f(y^{(i)})-\smax(y^{(i)})\right] - \left[f(y^{(j)})-\smax(y^{(j)})\right] - 2\left(1-\frac{j}{i}\right)^2.
    \end{equation*}

    After fixing the values of $y^{(j)}$ and (via asymptotic reasoning) values of gradients $\nabla f(y^{(j)})$, both sides of the identity~\eqref{eq:target-identity} are affine in $f_j = f(y^{(j)})$. So a collection of multipliers $\lambda$ satisfies~\eqref{eq:target-identity} if and only if the coefficients agree on each term: Namely, for the coefficients on $f_j$ to match, we need
    \begin{align*}
        \sum_{i=1}^{j-1} \lambda_{j,i} - \sum_{i=j+1}^d \lambda_{i,j} &= \begin{cases} 
            1 & \text{if } j=d, \\
            -1 & \text{if } j=1, \\
            0 & \text{otherwise},
        \end{cases}
    \end{align*}
    and then, for the constant terms to match,
    \begin{align*}        
        \sum_{i>j} \lambda_{i,j} \left(1-\frac{j}{i}\right)^2 &= \gamma.
    \end{align*}
    
    Consequently, calculation of the strongest lower bound, proven by this structured template, amounts to optimizing over a polyhedron: maximize the induced value of $\gamma$ over the set of nonnegative $\lambda$ with the needed identity. This polyhedron can be viewed combinatorially as routing one unit of flow from $1$ to $d$ with $\lambda_{i,j}$ flowing from $j$ to $i$. Our definition of $\gamma(d)$ in~\eqref{eq:gamma-definition} is a discrete formulation, where each partition $1=j_0< j_1<\dots< j_{k-1}< j_k = d$ corresponds to the extreme point with $\lambda_{j_{\ell+1},j_{\ell}}=1$ for $\ell=0,\dots k-1$ and all other multipliers equal to zero.

    From Theorems~\ref{thm:main_result} and~\ref{thm:exact_low_dim}, this lower bounding approach is quite strong, being within a small factor of the known upper bound and exact for $d=2$ and $d=3$. Preliminary numerics with $d= 4$ using sample points beyond $y^{(1)},\dots,y^{(4)}$ may provide stronger lower bounds. However, for such selections, we cannot use symmetry to determine exact gradients, making the resulting optimization much less tractable. We expect tighter bounds to follow from identifying new sample points with tractable proof-optimizations.
    }

    \section{Proof of Theorem~\ref{thm:main_result} -- An Elementary Lower Bound} \label{sec:proof-main}

    {\color{blue}We prove the first result of $\delta^\star(d) \geq \gamma(d)$ more generally than the case of $\infty$-norms of interest to this paper, deriving a bound for any $p$-norm smoothing, where $p \geq 1$. Letting $q$ denote the dual norm exponent (i.e., $1/p+1/q = 1$), a function is $1$-smooth with respect to $\|\cdot\|_p$ if $\|\nabla f(x) - \nabla f(y)\|_q \leq \|x-y\|_p$. Correspondingly, for any $1 \leq p \leq \infty$, define
    $$ \delta^\star_p(d) \defeq \begin{cases} \min_f &\|f-\smax\|_\infty\\
        \mathrm{s.t.} & f\colon \mathbb{R}^d\rightarrow \mathbb{R} \text{ is convex and $1$-smooth in $\|\cdot\|_p$} \end{cases}$$
    and for $p > 1$, define
    $$ \gamma_p(d) \defeq \begin{cases} \max & \sum_{\ell = 1}^{k}\frac{1}{4}\left(j_{\ell-1}\left(\frac{1}{j_{\ell-1}} - \frac{1}{j_{\ell}}\right)^q + (j_{\ell} - j_{\ell-1})\left(\frac{1}{j_{\ell}}\right)^q\right)^{2/q}\\
        \mathrm{s.t. }& 1=j_0< j_1<\dots< j_{k-1}< j_k = d\end{cases} $$
    with $\gamma_1(d)$ defined by
    $$
    \gamma_1(d) \defeq \begin{cases} \max & \sum_{\ell = 1}^{k}\frac{1}{4}\left(\max\left\{\tfrac{1}{j_{\ell -1}} - \tfrac{1}{j_{\ell}}, \tfrac{1}{j_{\ell}} \right\}\right)^{2}\\
        \mathrm{s.t. }& 1=j_0< j_1<\dots< j_{k-1}< j_k = d\end{cases}
    $$
    Notice that $\gamma_p(d)$ converges monotonically to $\gamma_1(d)$ and to $\gamma_\infty(d)=\gamma(d)$ as $p$ tends to one and infinity, respectively, (and so conversely, $q$ tends to infinity and one).
    }
    As a first observation, we note that without loss of generality, it suffices to restrict our attention to functions $f$ that are invariant under permutations. That is, functions with $f(x)=f(Px)$ for all $x\in\mathbb{R}^d$ and $P\in\Pi_d$ where $\Pi_d$ is the set of all $d\times d$ permutation matrices.
    \begin{lemma}\label{lem: convexityAndPermInvarianceIsHarmless}
        For any dimension $d\geq 2$, one has
        $$ \delta^\star_p(d) = \hat \delta^\star_p(d) \defeq \begin{cases} \min_f &\|f-\smax\|_\infty\\
        \mathrm{s.t.} & f\colon \mathbb{R}^d\rightarrow \mathbb{R} \text{ is convex and $1$-smooth in $\|\cdot\|_p$}\\
        & f(x) = f(Px) \ \forall x\in\mathbb{R}^d, P\in\Pi_d.
        \end{cases}$$
    \end{lemma}
    \begin{proof}[Proof of Lemma~\ref{lem: convexityAndPermInvarianceIsHarmless}]
        Clearly $\delta^\star_p(d) \leq \hat \delta^\star_p(d)$ as it has further constrained the search for an optimal smoothing. We establish the reverse direction by showing that for any feasible $f$ for $\delta^\star_p(d)$, a permutation invariant, convex, $1$-smooth in the $p$-norm $\hat f$ can be constructed that is no farther from $\smax$.
        To this end, consider any convex, $1$-smooth function $f$ with respect to $\|\cdot\|_p$. 
        Define
        $$\hat f(x) =  \frac{1}{d!}\sum_{P \in \Pi_d} f(Px).$$
        It is clear that for each $P \in \Pi_d$, the map $x \mapsto f(Px)$ is convex and $1$-smooth in the $p$-norm.\footnote{Applying the chain rule and the simple observation $\norm{P}_{p \to p} = 1$ gives the smoothness result.} As a result, $\hat f$ is convex and $1$-smooth in the $p$-norm. Finally, we establish $\hat{f}$ is no farther from $\smax$ than $f$ is as
        \begin{align*}
            \|f - \smax\|_{\infty} = \sup_{x \in \R^d} |f(x) - \smax(x)| &= \sup_{x \in \R^d}\sup_{P \in \Pi_d} |f(Px) - \smax(Px)| \\
            &\geq \sup_{x \in \R^d}\left|\frac{1}{d!}\sum_{P \in \Pi_d} f(Px) - \smax(Px) \right|\\
            &= \|\hat f - \smax\|_{\infty}
        \end{align*}
        where in the very last equality we have used the permutation invariance of $\smax$.
    \end{proof}
    
    Next, we show that the subdifferential $\partial f(x) \defeq \{\zeta \mid f(y) \geq f(x) + \langle \zeta, y-x\rangle \ \forall y\}$ of any convex approximator must lie in the simplex $\simplex{d} \subset \R^d$.
    \begin{lemma}\label{lem: subgradientsInSimplex}
        Let $f: \R^d \to \R$ be convex. If $\|f - \smax\|_{\infty} < \infty$, then for any $x \in \R^d$ and $\zeta \in \partial f(x)$, it holds that $\smax(y) \geq \ip{\zeta, y}$ for all $y \in \R^d$. In particular, $\partial f (x) \subseteq \simplex{d}$ for all $x \in \R^d$.
    \end{lemma}
    \begin{proof}[Proof of Lemma~\ref{lem: subgradientsInSimplex}]
        Let $\delta = \|f - \smax\|_{\infty}$ and fix $\zeta \in \partial f(x)$ for some arbitrary $x \in \R^d$. For any $y \in \R^d$ and any $\varepsilon > 0$, we have
        \begin{align*}
            \smax(y) = \varepsilon \smax(y/\varepsilon) \geq \varepsilon [f(y/\varepsilon) - \delta] 
            & \geq \varepsilon [f(x) +\ip{\zeta, y/\varepsilon - x} - \delta] \\
            &= \ip{\zeta, y} + \varepsilon[f(x) - \ip{\zeta, x} - \delta]
        \end{align*}
        where the second inequality is the subgradient inequality. Taking the limit as $\varepsilon \to 0$ gives $\smax(y) \geq \ip{\zeta, y}$. Since $x$, $y$, and $\zeta \in \partial f(x)$ were all arbitrary, this proves the first statement. The second statement follows directly from the first.
    \end{proof}
    Now let $f\colon \mathbb{R}^d\rightarrow \mathbb{R}$ be a convex, permutation invariant function that is $1$-smooth in $\|\cdot\|_{p}$ and satisfies $\|f - \smax\|_{\infty} < \infty$. For $j = 1, \dots, d$, let $x^{(j)}\in\mathbb{R}^d$ denote the point $(1/j,\dots,1/j,0,\dots, 0)$ with first $j$ entries set equal to $1/j$. For any $\alpha > 0$ and $i, j = 1, \dots, d$, convexity and $p$-norm smoothness of $f$~\cite[Theorem 5.8]{beck2017first} implies 
    \begin{align}
    Q_{i,j}(\alpha) \defeq f(\alpha x^{(i)}) &- f(\alpha x^{(j)}) -\ip{\nabla f(\alpha x^{(j)}), \alpha x^{(i)} - \alpha x^{(j)}} \label{eq: QijDef}\\
    & - \frac{1}{2}\|\nabla f(\alpha x^{(i)}) - \nabla f(\alpha x^{(j)})\|^2_q \geq 0. \nonumber 
    \end{align}
    For $i > j \geq 1$, a direct calculation gives
    \begin{align}
        Q_{i,j}(\alpha)
        &= f(\alpha x^{(i)}) - f(\alpha x^{(j)}) - \alpha\ip{\nabla f(\alpha x^{(j)}), \frac{j}{i}x^{(j)} - x^{(j)}} \nonumber\\
        & \qquad - \frac{1}{2}\|\nabla f(\alpha x^{(i)}) - \nabla f(\alpha x^{(j)})\|^2_q  - \alpha \ip{\nabla f(\alpha x^{(j)}), x^{(i)} - \frac{j}{i}x^{(j)}}\nonumber\\
        &\leq f(\alpha x^{(i)}) - f(\alpha x^{(j)}) - \left(\frac{j}{i} - 1\right)\ip{\nabla f(\alpha x^{(j)}), \alpha x^{(j)}} \nonumber\\
        & \qquad - \frac{1}{2}\|\nabla f(\alpha x^{(i)}) - \nabla f(\alpha x^{(j)})\|^2_q \nonumber\\
        &= f(\alpha x^{(i)}) - \smax(\alpha x^{(i)}) - \left[f(\alpha x^{(j)}) - \smax(\alpha x^{(j)})\right] \nonumber\\
        & \qquad - \frac{1}{2}\|\nabla f(\alpha x^{(i)}) - \nabla f(\alpha x^{(j)})\|^2_q \nonumber\\
        & \qquad  - \left(1 -\frac{j}{i}\right)\left[\smax(\alpha x^{(j)}) - \ip{\nabla f(\alpha x^{(j)}), \alpha x^{(j)}}\right] \nonumber\\
        & \leq f(\alpha x^{(i)}) - \smax(\alpha x^{(i)}) - \left[f(\alpha x^{(j)}) - \smax(\alpha x^{(j)})\right] \label{eq: QijBound}\\
        & \qquad - \frac{1}{2}\|\nabla f(\alpha x^{(i)}) - \nabla f(\alpha x^{(j)})\|^2_q \nonumber
    \end{align}
    where the first inequality holds because $\ip{\nabla f(\alpha x^{(j)}), x^{(i)} - \frac{j}{i}x^{(j)}} \geq 0$ since the components of $x^{(i)} - \frac{j}{i}x^{(j)}$ are nonnegative and Lemma~\ref{lem: subgradientsInSimplex} ensures $\nabla f(\alpha x^{(j)}) \in \simplex{d}$, the second equality holds because $\frac{j}{i}\smax(\alpha x^{(j)}) = \alpha/i = \smax(\alpha x^{(i)})$, and the last inequality follows from Lemma~\ref{lem: subgradientsInSimplex} and the assumption that $i > j$. 
    Now fix an increasing sequence $j_0 = 1 < j_1 < \dots < j_k = d$. Note that 
    \begin{align*}
        0 
        &\leq \sum_{\ell = 1}^{k}Q_{j_{\ell}, j_{\ell - 1}}(\alpha)  \\
        & \leq f(\alpha x^{(d)}) - \smax(\alpha x^{(d)}) - [f(\alpha x^{(1)}) - \smax(\alpha x^{(1)})]\\
        & \qquad\qquad - \sum_{\ell = 1}^{k}\frac{1}{2}\|\nabla f(\alpha x^{(j_{\ell})}) - \nabla f(\alpha x^{(j_{\ell -1})})\|^2_q \\
        & \leq 2\|f - \smax\|_{\infty} - \sum_{\ell = 1}^{k}\frac{1}{2}\|\nabla f(\alpha x^{(j_{\ell})}) - \nabla f(\alpha x^{(j_{\ell -1})})\|^2_q
    \end{align*}
    where the first inequality is by \eqref{eq: QijDef}, the second inequality holds because of inequality~\eqref{eq: QijBound} and the fact that the terms  $f(\alpha x^{(j_{\ell})}) - \smax(\alpha x^{(j_{\ell})})$ and $\left[f(\alpha x^{(j_{\ell-1})}) - \smax(\alpha x^{(j_{\ell -1})})\right]$ telescope, while the last inequality follows directly from the definition of $\|f - \smax\|_{\infty}$. Rearranging and taking the limit as $\alpha$ tends to infinity gives
    \begin{align}
        \|f - \smax\|_{\infty} 
        & \geq \lim_{\alpha \to \infty}\sum_{\ell = 1}^{k}\frac{1}{4}\|\nabla f(\alpha x^{(j_{\ell})}) - \nabla f(\alpha x^{(j_{\ell -1})})\|^2_q. \label{eq: lowerBoundBeforeLimit}
    \end{align}
    Note that $\nabla f(Px) = P\nabla f(x)$ for all $P \in \Pi_d$ and $x \in \R^d$ because of the permutation invariance assumption. In particular, $\nabla f(\alpha x^{(j)}) = P\nabla f(\alpha x^{(j)})$ for all $P \in \Pi_d$ such that $P\alpha x^{(j)} = \alpha x^{(j)}$. Combining this with the fact that $\nabla f(\alpha x^{(j)}) \in \simplex{d}$, we get that $\nabla f(\alpha x^{(d)}) = x^{(d)}$ and
    \begin{equation}\label{eq: structureOfGradients}
        \nabla f(\alpha x^{(j)}) = \left(\lambda_j(\alpha), \dots, \lambda_j(\alpha), \frac{1 - j\lambda_j(\alpha)}{d - j}, \dots, \frac{1 - j\lambda_j(\alpha)}{d - j}\right) 
    \end{equation}
    for all $\alpha > 0$ and $j = 1, \dots, d-1$, where $\lambda_j(\alpha) \in [0, 1/j]$. Therefore $\lim_{\alpha \to \infty} \nabla f(\alpha x^{(d)}) = x^{(d)}$. Similarly, $\lim_{\alpha \to \infty} \nabla f(\alpha x^{(j)}) = x^{(j)}$ for all $j = 1, \dots, d-1$, as the following calculation shows:
    \begin{align}
        0 \leq \lim_{\alpha \to \infty} \frac{1}{j} - \lambda_j(\alpha) &= \lim_{\alpha \to \infty} \frac{1}{\alpha}\left(\frac{\alpha}{j} - \alpha \lambda_j(\alpha)\right) \nonumber \\
        &= \lim_{\alpha \to \infty} \frac{1}{\alpha}\left(\smax(\alpha x^{(j)}) - \ip{\nabla f(\alpha x^{(j)}), \alpha x^{(j)}}\right) \nonumber \\
        & \leq \lim_{\alpha \to \infty} \frac{1}{\alpha}\left(\|f - \smax\|_{\infty}+ f(\alpha x^{(j)}) - \ip{\nabla f(\alpha x^{(j)}), \alpha x^{(j)}}\right)  \nonumber \\
        & \leq \lim_{\alpha \to \infty} \frac{1}{\alpha}\left(\|f - \smax\|_{\infty} + f(0)\right) \nonumber\\
        & = 0 \label{eq:gradient-derivation}
    \end{align}
    where the first inequality uses $\lambda_j(\alpha) \in [0, 1/j]$, the second equality holds by equation~\eqref{eq: structureOfGradients} and the definition of $x^{(j)}$, the second inequality uses $\smax(\alpha x^{(j)}) - f(\alpha x^{(j)}) \leq \|f - \smax\|_{\infty}$, and the last inequality uses the subgradient inequality $f(0) \geq f(\alpha x^{(j)}) + \ip{\nabla f(\alpha x^{(j)}), 0 - \alpha x^{(j)}}$. {\color{blue}As a result, for all $1 \leq q \leq \infty$, we have
    \begin{align*}
        &\|f - \smax\|_{\infty}\\
        &\geq \sum_{\ell = 1}^{k}\frac{1}{4}\|\lim_{\alpha \to \infty}\nabla f(\alpha x^{(j_{\ell})}) - \lim_{\alpha \to \infty}\nabla f(\alpha x^{(j_{\ell -1})})\|^2_q\\
        &= \sum_{\ell = 1}^{k}\frac{1}{4}\|x^{(j_{\ell})} - x^{(j_{\ell -1})}\|^2_q \\
        &=
        \begin{cases}
           \sum_{\ell = 1}^{k}\frac{1}{4}\left(j_{\ell-1}\left(\frac{1}{j_{\ell-1}} - \frac{1}{j_{\ell}}\right)^q + (j_{\ell} - j_{\ell-1})\left(\frac{1}{j_{\ell}}\right)^q\right)^{2/q} & \text{if } 1 \leq q < \infty \\
           \sum_{\ell = 1}^{k}\frac{1}{4}\left(\max\left\{\tfrac{1}{j_{\ell -1}} - \tfrac{1}{j_{\ell}}, \tfrac{1}{j_{\ell}} \right\}\right)^{2} & \text{if } q = \infty 
        \end{cases}
    \end{align*}
    where the inequality follows from \eqref{eq: lowerBoundBeforeLimit}. Since the sequence $j_0 = 1 < j_1 < \dots < j_k = d$ was arbitrary, maximizing over the sequences gives $\|f - \smax\|_{\infty} \geq \gamma_p(d)$ for all $1 \leq p \leq \infty$ and for all $d$.
    Hence $\delta^\star_p(d) \geq \gamma_p(d)$ for all $d$. When $p=\infty$ and $q=1$, the objective reduces to $\sum_{\ell = 1}^{k}\left(1 - \frac{j_{\ell-1}}{j_{\ell}}\right)^2$, giving the theorem's first claim.}

    Finally, we prove our asymptotic claim for $p=\infty$ by showing that
    \begin{equation}\label{eq:bounds}
        2\beta(1-\beta) \ln(d) - \frac{2(d-1)}{d} \leq \gamma(d) \leq 2\beta(1-\beta)\ln(d)
    \end{equation}
    where $\beta$ is the unique solution to $2\beta \ln \beta - \beta + 1 = 0$ in $(0,1)$. {\color{blue} (Table~\ref{tab:calculated-values} provides numerical values of $\gamma(d)$ and these bounds, where we see $\gamma(d)$ closely follows this upper bound.)} Numerically $\beta \approx 0.28467$, implying that $\gamma(d) \sim 2\beta(1-\beta)\ln(d) \approx 0.40726 \ln(d)$. The key observation in showing this is to note that $\gamma(d)$ is characterized by the recursive relation
    $$\gamma(d) = \max_{1 \le i < d} \left\{ \gamma(i) + \left(1-\frac{i}{d}\right)^2 \right\}, \qquad \gamma(1)=0.$$
    We prove both the upper and lower bounds inductively. At $d=1$ and $d=2$, one can verify directly.    
    The upper bound then follows inductively for $d>2$ by considering a continuous relaxation as
    \begin{align*}
        \gamma(d) 
        &\leq \max_{1 \le i < d} \left\{ 2\beta(1-\beta)\ln(i) + \left(1-\frac{i}{d}\right)^2 \right\}\\
        &\leq \max_{\phi\in (0,1)} \left\{ 2\beta(1-\beta)\ln(\phi) + \left(1-\phi\right)^2 \right\} + 2\beta(1-\beta)\ln(d) \\
        &= \left(2\beta(1-\beta)\ln(\beta) + \left(1-\beta\right)^2\right) + 2\beta(1-\beta)\ln(d)\\
        & {\color{blue}=} 2\beta(1-\beta)\ln(d)
    \end{align*}
    where the first line applies our inductive hypothesis, the second substitutes $i=\phi d$, the third notes that $\phi=\beta$ optimizes the expression, and the final line uses the definition of $\beta$.
    The lower bound follows inductively for $d>2$ by considering the simple selection $i=\lceil\beta d\rceil$ as
    \begin{align*}
        \gamma(d) &\geq  \gamma(\lceil \beta d \rceil) + \left(1-\frac{\lceil \beta d \rceil}{d}\right)^2 \\
        &\geq  2\beta(1-\beta)\ln(\lceil \beta d \rceil)  + \left(1-\frac{\lceil \beta d \rceil}{d}\right)^2 - \frac{2(\lceil \beta d \rceil - 1)}{\lceil \beta d \rceil}\\
        &\geq 2\beta(1-\beta)\ln(d) + 2\beta(1-\beta)\ln(\beta) + \left(1-\frac{\beta d +1}{d}\right)^2 - \frac{2\beta d}{\beta d + 1} \\
        &= 2\beta(1-\beta)\ln(d) - \frac{2(1-\beta)}{d} + \frac{1}{d^2} - \frac{2\beta d}{\beta d + 1}\\
        &\geq 2\beta(1-\beta)\ln(d) - \frac{2(d-1)}{d}
    \end{align*}
    where the second line applies our inductive hypothesis, the third applies bounds of $\beta d \leq \lceil\beta d\rceil \leq \beta d +1$, the fourth cancels terms using the definition of $\beta$, and the final collects and simplifies terms. \hfill $\Box$

    \begin{table}[t]
        \centering
        \caption{{\color{blue}Numerical values of $\gamma(d)$ and the bounds in~\eqref{eq:bounds}, rounded to five digits.}}
        \label{tab:calculated-values}
        \begin{tabular}{rccc}
            \toprule
            $d$
            & $2\beta(1-\beta)\ln(d)-\frac{2(d-1)}{d}$
            & $\gamma(d)$
            & $2\beta(1-\beta)\ln(d)$ \\
            \midrule
            $2$     & $-0.71771$ & $0.25000$ & $0.28229$ \\
            $3$     & $-0.88591$ & $0.44444$ & $0.44743$ \\
            $4$     & $-0.93541$ & $0.56250$ & $0.56459$ \\
            $5$     & $-0.94453$ & $0.64000$ & $0.65547$ \\
            $10$    & $-0.86224$ & $0.93444$ & $0.93776$ \\
            $100$   & $-0.10448$ & $1.86950$ & $1.87552$ \\
            $1000$  & $0.81528$  & $2.80460$ & $2.81328$ \\
            $10000$ & $1.75124$  & $3.74724$ & $3.75104$ \\
            \bottomrule
        \end{tabular}
    \end{table}

    \section{Proof of Theorem~\ref{thm:exact_low_dim} -- {\color{blue} Improved} Smoothings}

    {\color{blue} We prove the strict inequality $\delta^\star(d) < 0.5\ln(d)$ for $d\geq 3$ by demonstrating that a strictly better smoothing than LogSumExp exists with
    $$ \norm{f - \smax}_\infty \leq \frac{1}{2}\left[\ln(d-1) + \tfrac{1}{d(d-2)}\ln(d-1)\right] < \frac{1}{2}\ln(d).$$
    The strict inequality above follows by noting that 
    \begin{align}
        \ln(d) &= \ln(d-1) + \int_{d-1}^d\frac{1}{t}\, \mathrm{dt} \nonumber \\
        &> \ln(d-1) + \frac{1}{d} \nonumber \\
        &\geq  \ln(d-1) + \frac{1}{d(d-2)}\ln(d-1) \label{eq:needed-strict-ineq}
    \end{align}
    where the strict inequality uniformly lower bounds the integrand and the second inequality uses that $\ln(1 + t) \leq t$ at $t=d-2$.
    
    To construct this improvement, consider Nesterov-style smoothings
    \begin{equation}\label{eq: affineEntropyConjugate}
        f_{A, B}(x) \defeq \sup_{\lambda \in \Delta_d}\langle \lambda, x\rangle - h_{A, B}(\lambda).
    \end{equation}
    Here the dual function $h_{A,B}$ is defined as the closed convex proper function\footnote{Here we use the convention that $0\ln(0) = 0$.}
    \begin{equation}\label{eq: affineEntropy}
        h_{A,B}(\lambda) \defeq 
        \begin{cases}
            \frac{1}{A^2} \sum_{i=1}^d (A\lambda_i + B) \ln(A\lambda_i + B) - C& \text{if } \lambda \in \Delta_d \\
            +\infty & \text{otherwise}
        \end{cases}
    \end{equation}
    for any $A>0, B \geq 0$ with $A+dB=1$ and additive constant $C$ defined as
    $$ C = \frac{1}{2}(U_{A,B} + L_{A,B})$$
    where $U_{A,B} = \tfrac{(A + B)\ln(A + B) + (d-1)B\ln(B)}{A^2}$ and $L_{A,B} = \tfrac{A + dB}{A^2}\ln(\tfrac{A}{d} + B)$.
    
    Note that $f_{1, 0} = f_\mathtt{LSE} - \tfrac{1}{2}\ln(d)$, so this family includes LogSumExp (centered) as a special case. The following lemma allows us to select $A$ and $B$ more judiciously, yielding the claimed strictly better smoothing.
    \begin{lemma}\label{lem: parametrizedEntropy}
        For any $d \geq 2$ and $A>0,B\geq 0$ with $A+dB=1$, the function $f_{A,B}$ in~\eqref{eq: affineEntropyConjugate} is $1$-smooth in $\norm{\cdot}_{\infty}$ and 
        \begin{align*}
        \norm{f_{A, B} - \smax}_\infty &\leq \frac{1}{2}(U_{A,B} - L_{A,B}).
        \end{align*}
    \end{lemma}
    \begin{proof}[Proof of Lemma~\ref{lem: parametrizedEntropy}]
        First, we verify smoothness by showing $h_{A,B}$ is $1$-strongly convex with respect to $\norm{\cdot}_1$. We do this by reducing to the $1$-strong convexity of the entropy $h(z) = \sum_{i=1}^d z_i\ln(z_i)$ for $z\in\Delta_d$ (and $+\infty$ otherwise). That is, we know that for all $z,z'\in\Delta_d$ and $t\in[0,1]$,
        $$ h(t z + (1-t)z') \leq t h(z) + (1-t)h(z') - \frac{1}{2}t(1-t)\|z-z'\|^2_1.$$
        For any $\lambda,\lambda'\in\Delta_d$, consider the translated values with $z=A\lambda + B e$ and $z'=A\lambda' + B e$ where $e$ denotes the all ones vector. Since $A+Bd=1$, we have $z,z'\in\Delta_d$. Then dividing the above inequality by $A^2$ and subtracting $C$ from both sides, we have 
        $$ h_{A,B}(t \lambda + (1-t)\lambda') \leq t h_{A,B}(\lambda) + (1-t)h_{A,B}(\lambda') - \frac{1}{2}t(1-t)\|\lambda-\lambda'\|^2_1. $$
        Hence, $h_{A,B}$ is $1$-strongly convex with respect to $\norm{\cdot}_1$.
        
        Next, we bound $\norm{f_{A, B} - \smax}_\infty$. Observe that for any $x \in \R^d$, the difference $f_{A, B}(x) - \smax(x)$ is bounded above and below by
        \begin{align*}
            -\sup_{\hat\lambda \in \Delta_d} h_{A, B}(\hat\lambda) &= \sup_{\lambda \in \Delta_d}\ip{\lambda, x} -\sup_{\hat\lambda \in \Delta_d} h_{A, B}(\hat\lambda)  - \smax(x)\\
            &\leq f_{A, B}(x) - \smax(x) \\
            &\leq \sup_{\lambda \in \Delta_d}\ip{\lambda, x} -\inf_{\hat\lambda \in \Delta_d} h_{A, B}(\hat\lambda)  - \smax(x) \\
            &= - \inf_{\hat\lambda \in \Delta_d} h_{A, B}(\hat\lambda)
        \end{align*}
        where the equalities use the fact that $\smax(x) = \sup_{\lambda \in \Delta_d}\ip{\lambda, x}$, while the inequalities follow directly from the definition of $f_{A,B}$. This implies that
        \begin{align*}
            \norm{f_{A, B} - \smax}_\infty &\leq \max \left\{\sup_{\lambda \in \Delta_d} h_{A, B}(\lambda), -\inf_{\lambda \in \Delta_d} h_{A, B}(\lambda)\right\}.
        \end{align*}
        Now, since $h_{A, B}$ is convex and permutation invariant, the maximum over $\Delta_d$ is attained at the vertices of $\Delta_d$, while the minimum is attained at $(\tfrac{1}{d}, \dots, \tfrac{1}{d})$. Plugging in these $h_{A,B}$ values, the maximum value is $U_{A,B} - C$ and the minimum value is $L_{A,B} -C$. Simplifying terms, both components of the above maximum are equal to the claimed bound of $\frac{1}{2}(U_{A,B} - L_{A,B})$. \qedhere
    \end{proof}

    Consider the selection $A = \frac{d-2}{d-1}>0$ and $B = \frac{1}{d(d-1)}>0$. By the above lemma, the resulting $f_{A,B}$ provides a feasible smoothing, proving
    \begin{align*}
        \delta^\star(d) &\leq \norm{f_{A, B} - \smax}_\infty\\
        &\leq \frac{1}{2}\left[\frac{(A + B)\ln(A + B) + (d-1)B\ln(B)}{A^2} - \frac{A + dB}{A^2}\ln\left(\frac{A}{d} + B\right)\right]\\
        &= \frac{1}{2A^2}\left[\frac{d-1}{d}\ln\left(\frac{d-1}{d}\right) + \frac{1}{d}\ln\left(\frac{1}{d(d-1)}\right)- \ln\left(\frac{1}{d}\right)\right] \\
        &= \frac{1}{2A^2}\frac{d-2}{d}\ln(d-1)  \\
        &= \frac{1}{2}\left[\ln(d-1) + \frac{1}{d(d-2)}\ln(d-1)\right] < \frac{1}{2}\ln(d)
    \end{align*}
    where the final strict inequality is by~\eqref{eq:needed-strict-ineq}.
    }

    {\color{blue} Now we restrict to $d=2,3$ and provide a stronger, exactly optimal smoothing.} Note that $(1-1/d)^2$ is a lower bound on $\gamma(d)$, and hence $\delta^\star(d)$, for any $d$, established by considering Theorem~\ref{thm:main_result} with the simple sequence of length two $j_\ell: 1,d$. For $d=2,3$, we find that this lower bound is attained, i.e., $\delta^\star(d)=\gamma(d)=(1-1/d)^2$. In both cases, it suffices to consider a Nesterov-style smoothing using a quadratic distance function\footnote{For $d\geq 4$, this construction is no longer optimal. It is dominated by the LogSumExp construction in terms of smoothing gap $\|f-\smax\|_\infty$ and both fail to match our lower bound $\gamma(d)$.}
    $$f_d(x) = \sup_{\lambda \in \Delta_d} \langle \lambda, x\rangle - \frac{c_d}{2}(\|\lambda\|^2_2 - 1) - \gamma(d)$$
    where $\Delta_d$ is the simplex in $\mathbb{R}^d$ and $c_d = \max_{w\in\mathbb{R}^d\colon \sum w_i = 0} \|w\|^2_1/\|w\|^2_2$. 
    Here $c_d$ is chosen exactly to be the needed constant to ensure that $\frac{c_d}{2}(\|\lambda\|^2_2 - 1)$ is $1$-strongly convex in the one-norm on the simplex. Below, we briefly calculate $c_d$.
    
    By symmetry, we can assume that $w$ attaining the max has $k$ positive components taking value $a > 0$ and $d-k$ negative components taking value $-b < 0$ for some integer $1 \le k < d$. The constraint $\sum w_i = 0$ implies $ka - (d-k)b= 0$. We may scale $w$ without altering the ratio, so we choose $a = d-k$ and $b = k$.
    This yields a candidate vector $w^{(k)}$ with $k$ entries of $d-k$ and $d-k$ entries of $-k$. Note $\|w^{(k)}\|_1 = 2k(d-k)$ and $\|w^{(k)}\|_2^2 = dk(d-k)$.
    Substituting these into the objective function gives
    $$ \frac{\|w^{(k)}\|_1^2}{\|w^{(k)}\|_2^2} = \frac{4k^2(d-k)^2}{dk(d-k)} = \frac{4}{d} k(d-k). $$
    When $d$ is even, this is maximized by taking $k=d/2$, giving $c_d = d$. When $d$ is odd, this is maximized at both $k = (d \pm 1)/2$, giving $c_d = d - \frac{1}{d}$.

    Dually, the $1$-smoothness of this $f_d$ in the infinity-norm follows from the $1$-strong convexity of $\frac{c_d}{2}(\|\lambda\|^2_2 - 1)$ in the one-norm on the simplex. Similarly, the largest deviation between $f_d$ and $\smax$ is
    $$ \|f_d - \smax\|_\infty = \frac{1}{2}\left(\max_{\lambda\in\Delta_d} \frac{c_d}{2}\|\lambda\|^2_2 - \min_{\lambda\in\Delta_d} \frac{c_d}{2}\|\lambda\|^2_2\right).$$
    The maximum occurs at $\lambda=(1,0,\dots,0)$ and minimum occurs at $\lambda=(1/d,\dots,1/d)$, giving $\|f_d - \smax\|_\infty =  \frac{c_d}{4} (1-1/d)$. At $d=2$ and $d=3$, this equals $1/4$ and $4/9$, respectively, exactly agreeing with our lower bound of $(1-1/d)^2$. \hfill $\Box$

%% file: appendix.tex
\appendix

        \section{Proof of Proposition~\ref{prop:reduction_result} -- A Statistical Lower Bound Reduction}\label{app:stat_reduction}
    Here, we take the dual perspective of designing a strongly convex function $h$ on the simplex. Let $\mathrm{Range}(h) = \max_{\lambda\in \Delta_d} h(\lambda) - \min_{\lambda\in \Delta_d} h(\lambda)$ denotes the range of $h$ on the simplex. We claim that
    \begin{equation} \label{eq:dual-form-over-range}
        \delta^\star(d) = \begin{cases}
        \min & \frac{1}{2}\mathrm{Range}(h)\\
        \mathrm{s.t.} & h\colon \Delta_d \rightarrow \mathbb{R} \text{ is $1$-strongly convex in the one-norm}
        \end{cases}
    \end{equation}
    The claimed equivalence follows by considering the bijection $h=f^*$:

    This bijection maps between the two problems' domains as (i) a convex function $f\colon \mathbb{R}^d\rightarrow \mathbb{R}$ has $\|f-\smax\|_\infty < \infty$ if and only if $f^*$ has domain $\Delta_d$ and (ii) a function $f$ is $1$-smooth with respect to the infinity norm if and only if $f^*$ is $1$-strongly convex with respect to the one-norm by~\cite[Theorem 5.26]{beck2017first}. The equivalence of objective values under this bijection follows by considering the max function's representation $\smax(x) = \sup_{\lambda \in \Delta_d} \langle \lambda, x \rangle$ as $f(x) = \sup_{\lambda \in \Delta_d} \langle \lambda, x \rangle - h(\lambda)$ must have
    $$ \smax(x) - \max_{\lambda \in \Delta_d} h(\lambda) \le f(x) \le \smax(x) - \min_{\lambda \in \Delta_d} h(\lambda). $$
    Consequently, the deviation $\|f-\smax\|_\infty$ is minimized when the range of $h$ is centered around zero, yielding an error of $\frac{1}{2}(\max_{\lambda\in \Delta_d} h(\lambda) - \min_{\lambda\in \Delta_d} h(\lambda))$.

    From this, the claimed lower bound follows directly from the standard lower bounds on ``regret'' in online learning. Cesa-Bianchi et al.~\cite[Theorem 3.2.3]{Cesa1997} provide a statistical argument\footnote{Cesa-Bianchi et al.~\cite{Cesa1997}'s lower bound is proven by considering the task of predicting a series of $T$ coin flips with expert advice from $d$ experts. Online learning seeks to make predictions with minimal deviation from the performance of the best expert, i.e., ``regret''. If the series of coin flips is entirely fair, then no algorithm can outperform on average, while the best of $d$ random experts will statistically outperform the average.} establishing that no algorithm can achieve a regret bound lower than $(1+o(1))\sqrt{\ln(d)T/2}$. As an upper bound, Shalev-Shwartz~\cite[Theorem 2.11]{Shalev2012} establishes that for any $1$-strongly convex regularizer $h$ on the simplex, the ``Follow-the-Regularized-Leader'' algorithm has regret at most $\sqrt{2\mathrm{Range}(h)T}$. Combining these upper and lower bounds on regret, it follows that
    $$ (1+o(1))\sqrt{\ln(d)T/2} \leq \sqrt{2\mathrm{Range}(h)T} \ \iff \ \mathrm{Range}(h) \geq (1/4 + o(1)) \ln(d).$$
    Then from the dual equivalence~\eqref{eq:dual-form-over-range}, $\delta^\star(d) \geq (1/8 + o(1))\ln(d)$. \hfill $\Box$